\documentclass{amsart}
\usepackage{tikz-cd}
\usepackage{centernot}
\usepackage{xypic}
\usepackage{amssymb}
\usepackage{multirow}

\newcommand\BIT{\begin{enumerate}}
\newcommand\EIT{\end{enumerate}}

\newtheorem{thm}{Theorem}[section] 

\newtheorem{cor}[thm]{Corollary}

\newtheorem{exmpl}[thm]{Example}
\newtheorem{lem}[thm]{Lemma}
\newtheorem*{lem*}{Lemma}
\newtheorem{prop}[thm]{Proposition}

\newtheorem{rem}[thm]{Remark}

\newcommand{\Cref}[1]{{Corollary~\ref{#1}}}

\long\def\forget#1\forgotten{{}}

\def\Span{{\operatorname{Span}}}
\def\Ker{{\operatorname{Ker}}}

\begin{document}



\title[Amenable algebras and minimal subshifts]{Amenability of monomial algebras, minimal subshifts and free subalgebras}


\author{Jason P. Bell}
\address{Department of Mathematics, University of Waterloo, Waterloo, ON, N2L 3G1,
Canada}
\email{jpbell@uwaterloo.ca}

\author{Be'eri Greenfeld}
\address{Department of Mathematics, University of California, San Diego, La Jolla, CA, 92093, USA}
\email{bgreenfeld@ucsd.edu}

\begin{abstract}
    We give a combinatorial characterization of amenability of monomial algebras and prove the existence of monomial Følner sequences, answering a question due to Ceccherini-Silberstein and Samet-Vaillant. We then use our characterization to prove that over projectively simple monomial algebras, every module is exhaustively amenable; we conclude that convolution algebras of minimal subshifts admit the same property. We deduce that any minimal subshift of positive entropy gives rise to a graded algebra which does not satisfy an extension of Vershik's conjecture on amenable groups, proposed by Bartholdi. Finally, we show that non-amenable monomial algebras must contain noncommutative free subalgebras. Examples are given to emphasize the sharpness and necessity of the assumptions in our results.
\end{abstract}

\maketitle


\section{Introduction}

Amenability of discrete algebras has been extensively studied and connections to amenable groups, C*-algebras, dynamical systems, large-scale geometry, soficity and noncommutative algebra have led to many natural and interesting results in this vein. For a fair sample of the diverse spectrum of works, see Arzhantseva-Păunescu \cite{ArzhPau}, Bartholdi \cite{Bartholdi_book,Bartholdi_group algebra}, Ceccherini-Silberstein and Samet-Vaillant \cite{CSSV_JMS,CSSV_Expo}, Elek \cite{Elek}, Gromov \cite{Gro,Gromov}, 
and Ara-Li-Lied\'o-Wu \cite{Ara_BMS}.

We recall that for an algebra $A$ over a field $K$, a right $A$-module $M$ is \emph{amenable} if for every finite-dimensional subspace $V\leq A$ and every $\varepsilon>0$ there exists a finite-dimensional subspace $L\leq M$ (depending on $V$ and $\varepsilon$) such that
\begin{equation} \dim_K LV < (1+\varepsilon) \dim_K L.\end{equation}
In such a case, we say that $L$ is a $(V,\varepsilon)$-invariant subspace. 
Similarly, we say that $M$ is \emph{exhaustively amenable} if for each finite-dimensional subspace $V\leq A$, $\varepsilon>0$, and finite-dimensional subspace $W\leq M$ there exists a $(V,\varepsilon)$-invariant subspace of $M$ containing $W$.
The notions of amenable and exhaustively amenable for left modules are defined analogously.

It is worth making the remark that an infinite-dimensional $A$-module $M$ is exhaustively amenable if and only if for any finite-dimensional subspace $V\leq A$ and $\varepsilon>0$, $M$ has $(V,\varepsilon)$-invariant subspaces of arbitrarily large dimensions. An affine (that is, finitely generated over its base field) algebra $A=K\left<V\right>$ is exhaustively amenable as a module over itself if and only if it has a Følner sequence, namely a chain of finite-dimensional $K$-subspaces $L_1\subseteq L_2\subseteq\cdots$ such that $A=\bigcup_{n=1}^{\infty} L_n$ and $\dim_K L_nV/\dim_K L_n\rightarrow~1
$ as $n\to\infty$ (see \cite[Theorem~3.4]{CSSV_Expo}).


Bartholdi \cite{Bartholdi_group algebra} proved that a group $G$ is amenable if and only if its group algebra $K[G]$ is an amenable module over itself, if and only if every module over it is exhaustively amenable (for more on the equivalence of the various notions for group algebras, see \cite{Bartholdi_book}). A conjecture of Vershik \cite{Vershik} asserts that if $G$ is a finitely generated amenable group then the growth of the associated graded algebra $\bigoplus_{n=0}^{\infty} \varpi^n/\varpi^{n+1}$ is subexponential. Bartholdi \cite[Question~1.7]{Bartholdi2} (see also \cite{Bartholdi_talk}) has extended this conjecture to the setting of affine augmented algebras. He conjectured that if $A$ is an affine amenable algebra\footnote{The common definition of an amenable algebra in the literature is an algebra which is amenable, or exhaustively amenable as a module over itself. Bartholdi defines an amenable algebra to be an algebra over which every module is amenable; the latter definition has some significant advantages. To avoid confusion, from now on we will not use the term `amenable algebra' without explicitly mentioning the relevant property when it is clear from the context.} 
equipped with a homomorphism $\varepsilon\colon A\rightarrow K$ (called the augmentation map) then the associated graded algebra $\bigoplus_{n=0}^{\infty} \varpi^n/\varpi^{n+1}$, where $ \varpi=\Ker ( \varepsilon ) $, has subexponential growth.

The main algebraic objects of interest in this paper are monomial algebras; these are affine algebras, given by a presentation consisting solely of monomial relations. Monomial algebras are tightly connected to convolution algebras of \'etale groupoids arising from subshifts \cite{Nekrashevych}.

Our first main result is that algebras associated to minimal subshifts have the property that every module is exhaustively amenable.  (See \S\ref{sec:subshift} for relevant definitions concerning subshifts and associated algebras.) We recall that given a subshift $X$, one can naturally associate a monomial algebra $A_X$ that is spanned by the images of finite factors of $X$ and with multiplication given by concatenation (where a product of factors of $X$ which does not appear as a factor of $X$ is set to be $0$).  We then have the following result concerning amenability of algebras associated to minimal subshifts.
\begin{thm} \label{minimal amenable}
If $X$ is a minimal subshift then every (left or right) $A_X$-module is exhaustively amenable.
\end{thm}
Equivalently, this says that for a projectively simple monomial algebra, every (left or right) module is exhaustively amenable. Consequently, the monomial algebra spanned by the finite factors of every minimal subshift of positive entropy is a graded algebra of exponential growth all of whose modules are exhaustively amenable. This shows that Vershik's conjecture does not carry over to the setting of arbitrary augmented algebras (even under the strictest notion of amenability for algebras), answering Bartholdi's question.

Projectively simple monomial algebras can be viewed as certain subrings of convolution algebras of minimal subshifts, and we obtain results about the amenability of modules over the convolution algebra of the groupoid of the action of a minimal subshift (see \S\ref{sec:convo} for relevant definitions).
\begin{thm} \label{convo}
Let $X$ be a transitive aperiodic subshift. If $A_X$ is exhaustively amenable as a module over itself, then the convolution algebra, $K[\mathfrak{G}_X]$, of the groupoid 
of action is exhaustively amenable as a module over itself. If, moreover, $X$ is a minimal subshift, then every $K[\mathfrak{G}_X]$-module is exhaustively amenable.
\end{thm}

Theorems \ref{minimal amenable} and \ref{convo} can be thought of as algebraic counterparts of a result of Juschenko and Monod \cite{JusMon} that the topological full group of a minimal subshift is amenable.

To this end, we combinatorially characterize monomial algebras which are amenable or exhaustively amenable as modules over themselves (Theorem \ref{mon_amenable}, Lemma \ref{A/N}), and show that a monomial algebra that is amenable (resp.~exhausively amenable) as a module over itself admits monomial $(V,\varepsilon)$-invariant spaces (resp.~a monomial Følner sequence), thereby resolving a problem of Ceccherini-Silberstein and Samet-Vaillant\footnote{We note that Ceccherini-Silberstein and Samet-Vaillant use a different definition of amenability of an algebra than the one used by us; it is, however, not difficult to see that their definition is equivalent to what we call an algebra which is exhaustively amenable as a module over itself.} on the existence of monomial Følner sequences \cite{CSSV_Expo} (see also \cite[Page~85, Problem~14]{CSSV_JMS}).

Demonstrating the necessity of the minimality assumption to get a left-right symmetric amenabiliy result in Theorem \ref{minimal amenable}, we give an example of a transitive subshift without isolated points, whose associated monomial algebra 
is exhaustively amenable as a left, but not right module over itself. It follows that amenability transfers from a monomial algebra of a subshift to its convolution algebra, but is not inherited from a convolution algebra to its monomial subalgebra associated with the underlying subshift.

We then turn to investigate the relation between free subalgebras and amenability,
showing that a monomial algebra which does not contain a noncommutative free subalgebra generated by homogeneous elements is exhaustively amenable as a module over itself. 
\begin{thm} \label{no free amenable}
Let $A$ be a monomial algebra over an infinite field. If $A$ is not exhaustively amenable as a module over itself, then it contains a homogeneous free subalgebra.
\end{thm}

While an algebra of subexponential growth cannot contain a noncommutative free subalgebra, the converse is not true even for monomial algebras (see Proposition \ref{Example_10}). 

Our main results regarding monomial algebras can be summarized in the following diagram\footnote{The diagram refers to monomial algebras. The implication `No free subalgebras' $\implies$ `Exhaustively amenable over itself' works over any infinite field.}:

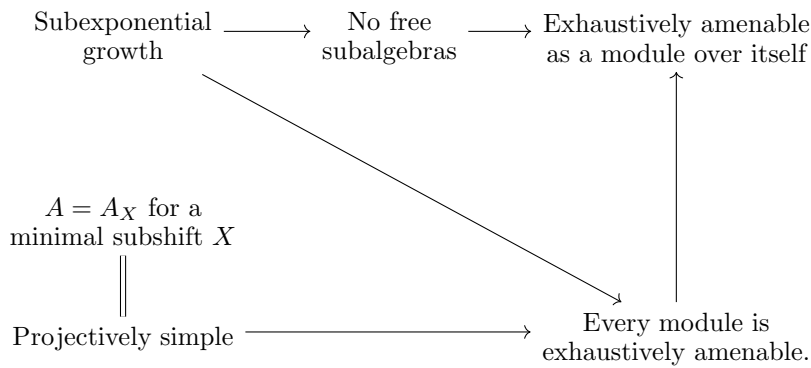
\[\begin{tikzcd}
	{\txt{Subexponential \\ growth}} & {\txt{No free \\ subalgebras}} & {\txt{Exhaustively amenable \\ as a module over itself}} \\
	& {} \\
	{\txt{$A=A_X$ for a \\ minimal subshift $X$}} & {} & {} \\
	{\txt{Projectively simple}} & {} & {\txt{Every module is \\ exhaustively amenable.}}
	\arrow[from=1-2, to=1-3]
	\arrow[from=1-1, to=1-2]
	\arrow[from=1-1, to=4-3]
	\arrow[from=4-3, to=1-3]
	\arrow[from=4-1, to=4-3]
	\arrow[from=4-1, to=3-1, Rightarrow, no head]
\end{tikzcd}\]
The outline of this paper is as follows.  In \S\ref{sec:subshift} we give background on subshifts and monomial algebras.  In \S\ref{sec:comb} we give a combinatorial characterization of amenability for monomial algebras and answer a question concerning isoperimetric profiles of general algebras.  In \S\ref{sec:amen} we give the proof of Theorem \ref{minimal amenable}, and we prove Theorem \ref{convo} in \S\ref{sec:convo}. Finally, we prove Theorem \ref{no free amenable} in \S\ref{sec:free}.

\textit{Conventions and terminology}. 
The various definitions for amenable \textit{algebras} amount to three: algebras that are amenable modules over themselves \cite{Ara_BMS,CSSV_JMS,CSSV_Expo}; algebras that are exhaustively amenable modules over themselves \cite{Elek,Gromov} (in \cite{Ara_BMS} these are called `properly amenable'); and algebras over which every module is amenable \cite{Bartholdi_book}.
Algebras are associative (but not necessarily commutative) and unital. Modules are right modules by default, unless otherwise stated or clear from the context.



\section{Monomial algebras and subshifts}\label{sec:subshift}

Let $K$ be an arbitrary field. A monomial algebra is an affine $K$-algebra given by a presentation consisting solely of monomial relations; that is:
$$ A\cong K\left<x_1,\dots,x_d\right>/I,$$ where $I$ is an ideal of the free algebra $K\left<x_1,\dots,x_d\right>$, generated by elements of the form $x_{i_1}\cdots x_{i_r}$, called monomials. Monomial algebras are graded by total degree of monomials: $A=\bigoplus_{i=0}^{\infty} A_i$ where $A_i$ is the space spanned by length-$i$ monomials in the generators. The ideal $\bigoplus_{i=1}^{\infty} A_i$ is the augmentation ideal of $A$. We say that a non-zero monomial $u$ in a monomial algebra is \emph{right prolongable} if there is a non-empty monomial $v$ such that $uv\neq 0$ (left prolongability is defined similarly), and \textit{prolongable} if it is both left and right prolongable. We say that $A$ is (left/right) \emph{prolongable} if all of its non-zero monomials are (left/right) prolongable.

Let $\Sigma=\{x_1,\dots,x_d\}$ be a finite alphabet. A subshift $X\subseteq \Sigma^{\mathbb{Z}}$ is a (non-empty) closed, shift-invariant topological subspace. We say that $X$ is \textit{transitive} if it has a dense shift-orbit, and \textit{minimal} if it contains no non-empty proper subshifts. Denote the shift operator by $T$.

Let $W\in \Sigma^{\mathbb{Z}}$ be an infinite word. 
We denote by $W[i,j]$ the factor (namely, subword) of $W$ starting from the $i$-th position and ending at the $j$-th position.
We say that $W$ is \textit{recurrent} if it contains infinitely many occurrences of each finite factor, and \textit{uniformly recurrent} if every finite factor $u$ has some constant $N(u)$ such that every factor of $W$ of length at least $N(u)$ contains an occurrence of $u$.

Every infinite word $W$ gives rise to a subshift, by considering the closure of its shift-orbit $\overline{\{T^i(W)\}}_{i\in \mathbb{Z}}$; and every subshift $X$ gives rise to a monomial algebra $A_X$, spanned by the finite factors of $X$ with multiplication being concatenation (product of factors of $X$ which does not appear as a factor of $X$ is set to be $0$). 

These constructions induce correspondences between the following subclasses:

$$\left\{\txt{Infinite \\words}\right\}\longrightarrow 
\left\{\txt{Transitive \\ subshifts}\right\}\ \subset\  \left\{\txt{Subshifts}\right\} \longleftrightarrow \left\{\txt{Prolongable \\ monomial algebras}\right\}.$$

An algebra is \textit{prime} if the product of non-zero ideals in it is non-zero. Then we have the following bijective correspondences:


$$\left\{\txt{Recurrent \\ infinite words}\right\}\longrightarrow \left\{\txt{Transitive irreducible \\ subshifts}\right\} \longleftrightarrow \left\{\txt{Prime monomial\\ algebras}\right\}.$$

An infinite-dimensional graded algebra is \textit{projectively simple} if its homogeneous non-zero ideals are finite-codimensional over the base field. Such algebras are sometimes called \textit{just-infinite}, and they are always prime.
Projectively simple algebras appear in various algebro-geometric contexts (e.g., see \cite{ProjSimple}), in combinatorial ring theory (see, for example,~\cite{BBL}) and symbolic dynamics, as they occur as certain subalgebras of convolution algebras of (the groupoid of action of) minimal subshifts (see \cite{Nekrashevych}).
In fact, if $X$ is a transitive (aperiodic) subshift then the convolution algebra of the \'etale groupoid of the action is isomorphic to a suitable localization of $A_X$ at the shift element (which is represented in $A_X$ by the sum of all letters of the underlying alphabet).
We have the following bijective correspondences:

$$\left\{\txt{Uniformly recurrent \\ infinite words}\right\}\longrightarrow \left\{\txt{Minimal \\ subshifts}\right\} \longleftrightarrow \left\{\txt{Projectively simple \\ monomial algebras}\right\}.$$

\section{Amenability of monomial algebras}\label{sec:comb}
In this section we give a general characterization of amenability in the case of monomial algebras and give results concerning isoperimetric profiles.

\subsection{A combinatorial characterization of amenability}
The main goal of this subsection is to prove the following result.
\begin{thm} \label{mon_amenable}
Let $A=K\left<x_1,\dots,x_d\right>/I$ be a right prolongable monomial algebra. Then the following are equivalent:
\begin{enumerate}
    \item $A$ is amenable as a right module over itself;
    \item for every $D\in \mathbb{N}$ there exists a non-zero monomial $u\in A$ which has at most one length-$D$ monomial $v$ such that $uv\neq 0$;
    \item $A$ has a Følner sequence as a right module over itself, consisting of spaces spanned by monomials (with respect to any generating subspace);
    \item $A$ is exhaustively amenable as a right module over itself.
\end{enumerate}
\end{thm}

\begin{proof}

$(1)\implies (2)$: 
Assume, to the contrary, that $A$ is amenable as a right module over itself but for some $D$, for every monomial $u\in A$ there are distinct monomials $v_1,v_2$ of length $D$ such that $uv_1,uv_2\neq 0$. Let $S=\{v_1,\dots,v_r\}$ be the set of all non-zero length-$D$ monomials in $A$.

Let $V=\Span_K\{1,x_1,\dots,x_d\}$, then by the amenability assumption there exists a 
finite-dimensional subspace $L\leq A$, say, $\dim_K L=n$, such that $\dim_K LV < 2\dim_K 
L$. Fix a basis for $L$ consisting of elements with distinct leading monomials (this is 
always possible by induction): $$L=~\Span_K\{f_1,\dots,f_n\}.$$ Let $z_i$ be the leading 
monomial of $f_i$ for each $1\leq i\leq n$. Let $v_{i,1},v_{i,2}$ be distinct length-$D$ 
monomials such that $z_iv_{i,1}, z_iv_{i,2}\neq 0$. Since all $z_iv_{i,j}$ are non-zero, they 
are the leading monomials of $f_iv_{i,j}$ for $1\leq i\leq n,\ j\in \{1,2\}$, respectively. But 
these leading monomials are distinct, since if $z_iv_{i,j}=z_{i'}v_{i',j'}$ then they must have 
the same length, and since $|v_{i,j}|=|v_{i',j'}|=D$, we get that $z_i=z_{i'}$ and $v_{i,j}
=v_{i',j'}$, and by the way we have chosen $f_1,\dots,f_n$ we have $i=i'$, and then $v_{i,j}
=v_{i,j'}$ so $j=j'$. Therefore the set $$\{f_iv_{i,j}|\ 1\leq i\leq n,\ j\in \{1,2\}\}$$ is linearly 
independent, and consequently $\dim_K LV\geq 2\dim_K L$, a contradiction.

$(2)\implies (3)$: Now assume that for any $D$, there exists a non-zero monomial $u\in A$ such that there is at most one length-$D$ monomial $v$ such that $uv\neq 0$; denote the set of such monomials $u$ by $U(D)$. 


Let $V=\Span_K \{f_1,\dots,f_r\}\leq A$ be given, as well as some $\varepsilon > 0$. Let $R$ bound from above the lengths of the monomials on which $V$ is supported. Take $N>R/\varepsilon$. Let $u\in U(N+R)$. For each $0\leq i\leq N+R$, let $w_i$ be the unique length-$i$ monomial for which $uw_i\neq 0$ (so $w_0=1$).
Let $L=\Span_K\{uw_0,\dots,uw_N\}$. Notice that for any $0\leq i\leq N$ and for every $f\in A$ supported on monomials of length at most $R$, we have that $uw_if$ is a linear combination of monomials of the form $uw_j$ for $i\leq j\leq N+R$. Thus $$LV\subseteq L+\Span_K\{uw_{N+1},\dots,uw_{N+R}\}.$$ It follows that
$$ \dim_K LV \leq \dim_K L + R < (1+\varepsilon)\dim_K L, $$
and so $L$ is a $(V,\varepsilon)$-invariant subspace. Since the dimension of $L$ tends to infinity as $\varepsilon\rightarrow 0$, it is clear that we can get a Følner sequence of monomial subspaces.

$(3)\implies (4)\implies (1)$ holds by definition.


\end{proof}

\begin{lem} \label{A/N}
Let $A$ be a monomial algebra and let $S\subseteq A$ be the set of monomials which are not infinitely right prolongable. Then $N=\Span_K S\triangleleft A$ is an ideal and $A/N$ is a right prolongable monomial algebra. Moreover,
\begin{itemize}
    \item if $|S|=\infty$ then $A$ has an infinite-dimensional monomial invariant subspace and is thus exhaustively amenable as a module over itself;
    \item if $|S|<\infty$ then $A$ is exhaustively amenable as a module over itself if and only if $A/N$ is exhaustively amenable as a module over itself, if and only if both have a sequence of Følner spaces (each one as a module over itself) spanned by monomials.
\end{itemize}
\end{lem}
\begin{proof}
That $N\triangleleft A$ and $A/N$ is right prolongable are straightforward.
If $S$ is infinite, 
then the set $S_0\subseteq S$ of monomials which are not right prolongable (namely, even by one letter) is also infinite.
Now $N_0=\Span_S S_0$ forms an infinite-dimensional space all of whose subspaces are invariant for any $V\leq A$ and $\varepsilon>0$. Hence $A$ is exhaustively amenable as a module over itself.

Suppose that $|S|=s<\infty$. The claim becomes evident if $A/N$ is finite-dimensional, so assume otherwise.
That $A$ is exhaustively amenable over itself if and only if $A/N$ is exhaustively amenable over itself follows from \cite[Proposition~3.8]{Ara_BMS}. Since $A/N$ is prolongable, by Theorem \ref{mon_amenable} its exhaustive amenability over itself is equivalent to having monomial Følner sequences.


Given $V\leq A$, $\varepsilon>0$ and $d\in \mathbb{N}$, let $\overline{V}=(V+N)/N$ and pick a $(\overline{V},\varepsilon/2)$-invariant subspace $L\leq A/N$ spanned by monomials, with $$\dim_K L > \max\{2s/\varepsilon,d\}.$$ Consider $\widehat{L}=L+N\leq A$, clearly a monomial subspace (since $L$ and $N$ are). Then $\dim_K \widehat{L}\geq d$ and in addition:
\begin{align*}
\dim_K \widehat{L} V \leq & \dim_K L \overline{V} + s \\  \leq & (1+\varepsilon/2)\dim_K L + s \\  < & 
(1+\varepsilon)\dim_K L \\  \leq & (1+\varepsilon)\dim_K \widehat{L}.
\end{align*}
So $A$ has a monomial Følner sequence as a module over itself.
\end{proof}

\begin{cor} \label{Answer1}
Let $A$ be a monomial algebra which is amenable as a right module over itself. Then $A$ has monomial $(V,\varepsilon)$-invariant subspaces. Furthermore, if $A$ is exhaustively amenable as a right module over itself then it has a monomial Følner sequence.
\end{cor}

\begin{proof}
Let $S$ be the set of monomials in $A$ which are not infinitely right prolongable. 
If $|S|<\infty$, then we are done (in the exhaustively amenable case) by Lemma \ref{A/N}.
If $|S|=\infty$ then we are again done by Lemma \ref{A/N}, noticing that $N_0$ therein is spanned by monomials.

To complete the picture, in view of Theorem \ref{mon_amenable}, the claim on (non-exhaustive) amenability becomes relevant only when $0<|S|<\infty$. In this case, the space $N_0=\Span_K S_0$ spanned by the (non-empty) set of non-prolongable monomials $S_0\subseteq S$ is a $(V,\varepsilon)$-invariant subspace for all $V\leq A$ and $\varepsilon>0$.
\end{proof}

This affirmatively answers a question posed in \cite[Page 161]{CSSV_Expo}.

\subsection{Isoperimetric profiles}

Let $G$ be a group generated by a finite subset $S$. The \textit{isoperimetric profile} (defined by Gromov \cite{Groiso} and Vershik \cite{Vershik}) of $G$ with respect to $S$ is the function measuring the minimum $S$-boundary of a size-$n$ subset of $G$:

$$ I_{G,S}(n)=\inf_{|X|=n} |\partial_S(X)| $$
where the boundary of $X$ with respect to $S$ is $\partial_S(X)=\bigcup_{s\in S} Xs \setminus X$.

This is an important measurement in geometric group theory \cite{Ers,Gro,Groiso,Pit,Vershik,Zuk}. By analogy with the group-theoretic setting, Gromov \cite{Gro} defined and studied isoperimetric profiles of algebras, which were further studied by D'Adderio \cite{DAdderio}.
Let $A$ be an algebra generated by a finite-dimensional subspace $V$ and suppose $1\in V$. The isoperimetric profile of $A$ with respect to $V$ is the function:
$$ I_{A,V}(n) = \inf_{\substack{W\leq A \\ \dim_F W = n}} \{ \dim_F WV/W \}.$$
(See definitions in \cite[Page~180]{DAdderio}.) A finitely generated algebra is exhaustively amenable as a module over itself if and only if\footnote{We say that $f\preceq g$ if $\forall n,\ f(n)\leq c_1g(c_2n)$ for some $c_1,c_2>0$, and $f\prec g$ if $f\preceq g$ but $g\not\preceq f$. Two functions are asymptotically equivalent, $f \sim g$, if $f\preceq g$ and $g\preceq f$.} $I_{A,V}(n) \prec n$.

The proof of Theorem \ref{mon_amenable} immediately yields the following result.

\begin{cor}
Let $A$ be a prolongable monomial algebra. Then its isoperimetric profile is constant; that is,~$I_{A,V}\sim O(1)$. 
\end{cor}

D'Adderio defined (asymptotic) subadditivity as follows: a function $$f\colon \mathbb{R}_+\rightarrow \mathbb{R}_+$$ is subadditive if there exist $c_1,c_2>0$ such that $$c_1f(c_2(x_1+\cdots+x_r))\leq f(x_1)+\cdots+f(x_r)$$ for all $r$ and $x_1,\dots,x_r\in \mathbb{R}_+$. This is a version of classical subadditivity which is invariant under the standard notion of equivalence of functions in asymptotic group theory.

By analogy with the group-theoretic case, D'Adderio proved that isoperimetric profiles of domains are subadditive \cite[Theorem~2.3.5]{DAdderio} and asked whether this holds for general algebras \cite[Question~1]{DAdderio} (the isoperimetric profile is thought of as a positive real-valued function by $I_{A,V}(r)=I_{A,V}(\lfloor r \rfloor)$). We conclude this section by giving a construction that shows D'Adderio's question has a negative answer in general.

\begin{exmpl} The algebra $A$ constructed below is a finitely generated algebra whose isoperimetric profile is not asymptotically subadditive.
\end{exmpl}
Let $n_i=2^{2^{2^i}}$ for each $i\geq 3$. Consider the infinite direct product: $$ R = \prod_{i=3}^{\infty} M_{n_i}(F) $$
and let $e=(e_i)_{i\geq 3}\in R$ be the element whose components $e_i \in M_{n_i}(F)$ are the idempotent matrices with the upper left entry $1$ and all other entries zeros:
$$\left( \begin{array}{ccc}
1 & 0 & \cdots \\
0 & 0 & \cdots \\
\vdots & \vdots & \ddots
\end{array} \right).$$
Let $\sigma=(\sigma_i)_{i\geq 3}$ be the element whose components $\sigma_i \in M_{n_i}(F)$ are the permutation matrices: $$\left( \begin{array}{cc}
0 & I_{n_i-1} \\
1 & 0 \end{array} \right),$$ where $I_m$ denotes the $m$-by-$m$ identity matrix.
Inside $R$, let $A=F\left<e,\sigma\right>$; this algebra was introduced in \cite{ArbitrarySimpleModules}. Let $V=F+Fe+F\sigma$. Since for each $i$ the elements $e_i,\sigma_i$ generate the full matrix ring $M_{n_i}(F)$, it follows that $A$ is a subdirect product of matrix algebras. Moreover, as proved in \cite[Lemma~2.1]{ArbitrarySimpleModules}, the subspace $A_0=\bigoplus_{i\geq 3} M_{n_i}(F)$ of eventually zero sequences of matrices is contained in $A$.


Let $W\leq A$ be a finite-dimensional $V$-invariant subspace. Equivalently, $W$ is a right ideal (as $A=F\left<V\right>$). 
Clearly the projection of $W$ onto each component $M_{n_i}(F)$ is a right ideal (and hence is either zero-dimensional or has dimension at least $n_i$), and since $\dim_F W<\infty$ it follows that $W$ is supported on a finite number of components, say, $n_1,\dots,n_r$.
But since $A_0\subseteq A$ it follows that $W$ is a right ideal of $A_0$, so $W=\bigoplus_{i=1}^r V_i^{\oplus k_i}$ where each $V_i$ is a right ideal of $M_{n_i}(F)$ and $0\leq k_i\leq n_i$. Hence $\dim_F W=\sum_{i=1}^{r} k_i 2^{2^{2^i}}$. Let 
$$ \mathcal{S} = \Bigg\{\sum_{i=1}^{r} k_i 2^{2^{2^i}}\ |\ \forall 1\leq i\leq r,\ 0\leq k_i\leq 2^{2^{2^i}} \Bigg\}.$$ By the above, if $s\in \mathbb{N}\setminus \mathcal{S}$ then $I_{A,V}(s)>0$. Conversely, it is clear that every $s\in \mathcal{S}$ is equal to $\dim_F W$ for some right ideal $W\leq A$, and so
$$ I_{A,V}(s)=0 \iff s\in \mathcal{S}. $$
Notice that for each $i\geq 3$, the maximum number in $\mathcal{S}$ which is smaller than $n_{i+1}$ is $ n'_{i+1} := \sum_{j=1}^i n_j^2\leq in_i^2 < 2^{2^{2^i+2}}$. Thus we obtain the limit \begin{equation} \lim_{i\rightarrow \infty} n_{i+1}/n'_{i+1} = ~ \infty.
\label{eq:infty}
\end{equation}
Assume to the contrary that $I_{A,V}$ is asymptotically subadditive; in particular, for some $c_1,c_2>0$ we have $$c_1 I_{A,V}(\lfloor 2c_2x \rfloor)\leq 2I_{A,V}(x)$$ and $$c_1 I_{A,V}(\lfloor 3c_2x \rfloor)\leq 3I_{A,V}(x).$$  There are two cases to consider.
\vskip 2mm
\emph{Case 1}: $c_2<1/2$. In this case, take $x=n_i$ where $i\gg 1$ such that $\lfloor 2c_2 n_i \rfloor > n'_i$ (possible by Equation (\ref{eq:infty})). But $\lfloor 2c_2x \rfloor < x=n_i$ so $\lfloor 2c_2x \rfloor\notin \mathcal{S}$. Therefore $0 < c_1 I_{A,V}(\lfloor 2c_2x \rfloor)\leq 2I_{A,V}(x)=0$, a contradiction.
\vskip  2mm
\emph{Case 2}: $c_2>1/3$. Take $x=n'_i$ where $i\gg 1$ such that $\lfloor 3c_2n'_i \rfloor>n'_i$ and $\lfloor 3c_2n'_i \rfloor<n_i$ (again possible by Equation (\ref{eq:infty})). Then $\lfloor 3c_2x \rfloor\notin \mathcal{S}$, so $0 < c_1 I_{A,V}(\lfloor 3c_2x \rfloor)\leq 3I_{A,V}(x)=0$, a contradiction.

\section{Amenability of monomial algebras of minimal subshifts}\label{sec:amen}



\subsection{Minimal subshifts}
In this subsection our chief objective is to prove Theorem \ref{minimal amenable}.
The following combinatorial lemma plays a major role in the proof of this theorem.

\begin{lem} \label{combLemma}
Let $W$ be a uniformly recurrent word. Then for every $D\in \mathbb{N}$ there exists a factor $u$ of $W$ for which there is a unique word $v$ of length $D$ such that $uv$ is a factor of $W$.
\end{lem}

\begin{proof}
First, observe that the claim holds for $D=1$. If every factor in $W$ has at least two distinct (right) prolongations by a letter, we can construct an arbitrarily long factor avoiding an arbitrary letter from the alphabet, contradicting uniform recurrence.

Now, for $D>1$, assume to the contrary that $W$ is a uniformly recurrent word in which every factor has at least two distinct length-$D$ (right) prolongations. Let $u_1,\dots,u_m$ be the distinct length-$D$ factors of $W$. Without loss of generality, we may replace $W$ by a right-infinite word (this is ensured by recurrence) supported on $\mathbb{N}\cup \{0\}$, namely, $W=W[0]W[1]W[2]\cdots$.

For each $0\leq j \leq D-1$ we can uniquely decompose: $$ W = v_j u_{f(1,j)} u_{f(2,j)} u_{f(3,j)} \cdots $$
where $v_j$ is the length-$j$ prefix of $W$ and $f\colon \mathbb{N}\times \{0,1,\dots,D-1\}\rightarrow \{1,\dots,m\}$ is some function.

Let $\mathcal{T}=\{z_1,\dots,z_m\}$ be a new alphabet. Define right-infinite words over it as follows:
$$ W_j = z_{f(1,j)} z_{f(2,j)} \cdots \ \ \ \text{for each}\ 0\leq j \leq D-1 $$
\textit{Claim}: The words $W_0,\dots,W_{D-1}\in \mathcal{T}^{\mathbb{N}}$ are uniformly recurrent. \\
\textit{Proof of the claim}: Given a factor $v$ of $W$, we let: 
\begin{eqnarray*}
S(v) & = & \{p\in \mathbb{N}\ |\ W\ \text{has a copy of}\ v\ \text{starting at the}\  p\text{-th position} \} \\
S_0(v) & = & \{ p\in \{0,1,\dots, D-1\}\ |\ (p+D\mathbb{Z})\cap S(v)\neq \emptyset \}.
\end{eqnarray*}
Given a length-$Dn$ factor $v$ of $W$, we can write $v=u_{i_1}\cdots u_{i_n}$. Let $S_0(v)=\{p_1,\dots,p_r\}$ where $p_1<\cdots <p_r$.
Let $W'$ be a finite prefix of $W$ in which $v$ occurs in positions $p'_1,\dots,p'_r$ such that $p'_i\equiv p_i \mod D$ for all $1\leq i \leq r$. Since $W$ is uniformly recurrent, there exists $N$ such that every length-$ND$ factor of $W$ contains a copy of $W'$. It follows that every occurrence of $W'$ in $W$ contains occurrences of $v$ at the positions (in $W$) $p'_1+s,\dots,p'_r+s$ for some $s\geq 0$.
By definition, for each $1\leq i \leq r$ we have: $p'_i+s\equiv p'_{\sigma(i)} \mod D$ for some $\sigma\colon \{1,\dots, r\}\rightarrow \{1,\dots,r\}$. But clearly $\sigma$ is injective, hence a permutation.

Fix $0\leq j \leq D-1$ and let $z_{i_1}\cdots z_{i_n}$ be an arbitrary factor of $W_j$. Then there is some $l\geq 0$ such that: $$ W[j+lD,j+lD+nD-1] = u_{i_1}\cdots u_{i_n}, $$ call this factor (of $W$) $v$. Thus there exists $p_k\in S_0(v)$ such that $p_k\equiv j \mod D$.

Let $y$ be an arbitrary length-$N$ factor of $W_j$, starting at some position $q$, namely, $y=W_j[q,q+N-1]$. Consider: $$\hat{y}=W[j+qD,j+qD+ND-1]=u_{\alpha_1}\cdots u_{\alpha_N}$$

By definition of $N$, there is an occurrence of $W'$ in $\hat{y}$, and by the above argument (that $\sigma$ is a permutation) there is an occurrence of $v$ in $\hat{y}$ which starts at a position (in $W$) which is congruent to $p_k$, and hence to $j$, modulo $D$. It follows that:
$$ u_{\alpha_t}\cdots u_{\alpha_{t+n-1}} = v $$
for some $1\leq t \leq N-n+1$ and consequently:
$$ z_{\alpha_t}\cdots z_{\alpha_{t+n-1}} = z_{i_1}\cdots z_{i_n} $$
factors $y$, thus proving that $W_j$ is uniformly recurrent. The claim is now proven.

We now turn to complete the proof of the lemma. We assume to the contrary that $W$ is a uniformly recurrent word in which for some $D>1$ every factor has at least two distinct right length-$D$ prolongations.
Let $v=u_{i_1}\cdots u_{i_n}$ be a (length-$Dn$) factor of $W$ such that $S_0(v)$ is minimal with respect to inclusion among all factors of length divisible by $D$.

For each $1\leq k \leq m$, such that $vu_k$ factors $W$, we have $S_0(vu_k)\subseteq S_0(v)$ and by minimality, $S_0(vu_k)=S_0(v)$. Similarly, $S_0(u_kv)=S_0(v)$ (whenever $u_kv$ factors $W$).

Pick $j\in S_0(v)$. Since $W_j$ is uniformly recurrent (by the above claim), and by the validity of the lemma for $D=1$, there exists a factor $y$ of $W_j$ with a unique length-$1$ (over the alphabet $\mathcal{T}$) right prolongation.

Since $j\in S_0(v)$, there is an occurrence of $z_{i_1}\cdots z_{i_n}$ in $W_j$. By recurrence of $W_j$, it has a factor of the form: $$ z_{i_1}\cdots z_{i_n} y' y $$ for suitable $y'$. 
Since $y$ has a unique length-$1$ prolongation, then so does $ z_{i_1}\cdots z_{i_n} y' y $, say, \begin{equation} \label{eq:zz}
z_{i_1}\cdots z_{i_n} y' y z_c \end{equation} is its unique length-$1$ prolongation (for some $1\leq c \leq m$).
Consider the length-$D(n+|y'|+|y|)$ factor $g$ of $W$ corresponding to $ z_{i_1}\cdots z_{i_n} y' y $. Since $v$ is a prefix of $g$, we have $S_0(g)=S_0(v)$. By Equation (\ref{eq:zz}), we see that $gu_c$ is a length-$D$ prolongation of $g$.
But by assumption, there is at least one additional length-$D$ prolongation for $g$, say, $gu_d$ for some $1\leq d \leq m$ and $d\neq c$. We claim that $j\notin S_0(gu_d)$, for otherwise (replacing each $u_i$ by the corresponding $z_i$) we would get that $ z_{i_1}\cdots z_{i_n} y' y z_d$ factors $W_j$, contradicting the uniqueness in Equation (\ref{eq:zz}). Hence $j\in S_0(v)\setminus S_0(gu_d)$. Since: $$ S_0(gu_d)\subseteq S_0(g)=S_0(v) $$ we get a contradiction to the minimality of $S_0(v)$.
\end{proof}

\begin{proof}[{Proof of Theorem \ref{minimal amenable}}]
We prove the theorem for right modules, and the left module case is equivalent (using a left-wise counterpart of Lemma \ref{combLemma}). Let $X$ be a minimal subshift and let $$A=A_X=K\left<x_1,\dots,x_d\right>/I$$ be the corresponding monomial algebra. Let $A_i$ be the subspace spanned by length-$i$ monomials.
Let $M$ be a right $A$-module, $V\leq A$ a finite-dimensional subspace and $\varepsilon>0$. 
Let $R$ be an upper bound on the lengths of all monomials on which the elements of $V$ are supported. 

Let $C=\lceil \frac{R}{\varepsilon}\rceil$. Take $D=C+R$. By Lemma \ref{combLemma}, there exists a factor $u$ of $X$ with a unique length-$D$ right prolongation, say, $ux_{i_1}\cdots x_{i_D}$. By minimality of $X$, there exists some $N$ such that every length-$N$ monomial in $A$ is divisible by $ux_{i_1}\cdots x_{i_D}$. Given a length-$N$ monomial $v\in A$, let us write $v=v_0ux_{i_1}\cdots x_{i_D}v'$ for some $v_0,v'$.

\textit{Case 1}: For every $\varepsilon > 0$ there exist $\xi \in M$, and a monomial $v\in A$ of length $N$ such that the set:
$$ \{\xi\cdot v_0u,\xi\cdot v_0ux_{i_1},\dots,\xi\cdot v_0ux_{i_1}\cdots x_{i_C}\} $$ 
is linearly independent (that is, its span -- call it $L_C$ -- is $(C+1)$-dimensional).

Consider an arbitrary element $f\in V$. Then $f$ is a linear combination of monomials of length at most $R$, so for each $1\leq j\leq C$ we have that: $$ \xi\cdot v_0ux_{i_1}\cdots x_{i_j}f\in \Span_K \{\xi\cdot v_0u,\xi\cdot v_0ux_{i_1},\dots,\xi\cdot v_0ux_{i_1}\cdots x_{i_D}\} =: L_D, $$ (and also $\xi\cdot v_0uf \in L_D$) hence $L_CV\subseteq L_D$.
Thus:
\begin{eqnarray*}
\dim_K L_C V &  \leq & \dim_K L_D\leq D+1 \\ & = & C+R+1<(1+\varepsilon)(C+1) \\ & = & (1+\varepsilon)\dim_K L_C
\end{eqnarray*}
so there is a $(V,\varepsilon)$-invariant subspace. Moreover, $\dim_K L_C\xrightarrow{\varepsilon\rightarrow 0}\infty$, so we get exhaustive amenability of $M$.

\textit{Case 2}: There exists $\varepsilon > 0$ such that for every $\xi \in M$ and for every monomial $v\in A$ of length $N$, the set:
$$ \{\xi\cdot v_0u,\xi\cdot v_0ux_{i_1},\dots,\xi\cdot v_0ux_{i_1}\cdots x_{i_C}\} $$
is linearly dependent.
In particular, either $\xi\cdot v_0u=0$, so $\xi \cdot v=0$, or for some $1\leq j\leq C$, we can write:
$$ \xi\cdot v_0ux_{i_1}\cdots x_{i_j} = \xi \cdot \sum_{l=1}^{j-1} \gamma_l v_0ux_{i_1}\cdots x_{i_l} + \gamma'\xi\cdot v_0u. $$
So:
\begin{eqnarray*}
\xi\cdot v = \xi\cdot v_0ux_{i_1}\cdots x_{i_D}v' & = & \xi \cdot \sum_{l=1}^{j-1} \gamma_l v_0ux_{i_1}\cdots x_{i_l}x_{i_{j+1}}\cdots x_{i_D}v' \\ & + & \gamma'\xi\cdot v_0ux_{i_{j+1}}\cdots x_{i_D}v'
\end{eqnarray*}
for some scalars $\gamma',\gamma_1,\dots, \gamma_{j-1}$.
Hence: $$\xi\cdot v\in \xi\cdot \bigoplus_{i=0}^{N-1} A_i,$$ and it follows that $\xi\cdot A_N\subseteq \bigoplus_{i=0}^{N-1} \xi\cdot A_i$. It follows that the $A$-submodule of $M$ generated by any $0\neq \xi \in M$ is  a non-zero, finite-dimensional subspace, which is thus $(V,\varepsilon)$-invariant (for any $\varepsilon>0$).
It follows that $M$ is an exhaustively amenable right $A$-module.
\end{proof}

As an immediate consequence we obtain the following result.

\begin{cor} \label{cor}
Let $X$ be a minimal subshift of positive entropy. Then $A_X$ is a projectively simple, graded algebra of exponential growth, all of whose modules are exhaustively amenable.
\end{cor}

Notice that explicit minimal subshifts of positive entropy were constructed in \cite{Min_entropy}.

\subsection{Transitive subshifts with asymmetric amenability}

\begin{exmpl} The following gives a transitive subshift without isolated points whose associated monomial algebra is left exhaustively amenable, but not right amenable. \label{asymmetry}
\end{exmpl}
Consider the alphabet $\Sigma=\{w,x,y,z\}$. Let $\mathcal{L}\subset \Sigma^{*}$ be the language consisting of all words avoiding factors of the form:
$$ Vxy^{2i}x\ \text{where}\ V\ \text{is not a suffix of}\ wz^{2i}w\ \text{and}\ |V|\leq 2i+2 $$
for all $i\in \mathbb{N}$. Equivalently, if $U\in \mathcal{L}$ and $U=U_0xy^{2i}xU_1$ for some $i$ then either $U_0$ has $wz^{2i}w$ as a suffix or $|U_0|<2i+2$ and $U_0$ is a suffix of $wz^{2i}w$. Clearly $\mathcal{L}$ is a hereditary language.

First, we claim that $\mathcal{L}$ is the language of finite factors of some recurrent word. Let $U_1,U_2\in \mathcal{L}$. Consider the leftmost occurrence of $xy^{2i}x$ in $U_2$ (if it exists), then $U_2=U_3xy^{2i}xU_4$ where $U_3$ either has $wz^{2i}w$ as a suffix, or else $|U_3|<2i+2$ and $U_3$ is a suffix of $wz^{2i}w$. In the first case, define: $$ W=z^{|U_1|+|U_2|+1} $$ and in the latter case take $W$ such that $WU_3=wz^{2i}w$.

We claim that $U_1WU_2\in \mathcal{L}$. For if we had an occurrence of a forbidden factor of the form $Vxy^{2i}x$ as above within $U_1WU_2$, it must have occurred within $U_2$. Any such occurrence which is not the leftmost occurrence cannot be of the forbidden form.
Fix the leftmost such occurrence. Then by construction, it has an occurrence of $wz^{2i}w$ consecutively to its left, so again there is no forbidden factor.

It follows that there is a recurrent word $\Omega\in \Sigma^{\mathbb{Z}}$ whose set of finite factors coincides with $\mathcal{L}$. The closure of the shift-orbit of $\Omega$, say, $X$ is a transitive subshift without isolated points. Let $A=A_X$ be the monomial algebra associated with $X$.

Any factor $U$ of $\Omega$ can be prolonged to the right by both $w$ and $z$. Therefore, by Theorem \ref{mon_amenable}, $A$ is non-amenable as a right module over itself. But for any $D\in \mathbb{N}$, there exists a factor of $\Omega$ (namely, $xy^{2D}x$) which can be prolonged in a unique way to the left by a length-$D$ factor (since any length-$\leq 2D+2$ factor appearing consecutively to the left of $xy^{2D}x$ must be a suffix of $wz^{2D}w$), so $A$ is an exhaustively amenable left module over itself.


\begin{rem}
The above example contains a noncommutative free monomial subalgebra, e.g.~$K\left<x,xy\right>$. In Section \ref{sec:free} we shall see that conversely, if a monomial algebra contains no noncommutative free subalgebras then it is in fact an exhaustively amenable module over itself.
\end{rem}

\section{Convolution algebras}\label{sec:convo}

Let $X\subseteq \Sigma^\mathbb{Z}$ be a transitive aperiodic subshift, where $\Sigma=\{x_1,\dots,x_d\}$ is a fixed alphabet. Denote the shift operation by $T$.

\subsection{Convolution algebras and monomial subalgebras}
The groupoid $\mathfrak{G}_{X}$ of the action 
 $\mathbb{Z} \curvearrowright X$ is the groupoid whose elements are $\mathbb{Z}\times X$ and (partial) multiplication is given by:
$$(m_2, T^{m_1}(x))\cdot(m_1,x) = (m_1+m_2,x).$$
This is an \'etale Hausdorff groupoid whose space of units is totally disconnected.
One can associate with $\mathfrak{G}_X$ an associative algebra, which we denote $K[\mathfrak{G}_X]$, called the \textit{convolution algebra} of $\mathfrak{G}_X$, consisting of all continuous, compactly supported functions $f\colon\mathfrak{G}_X\rightarrow K$ (here the base field $K$ is arbitrary and endowed with the discrete topology). The multiplicative structure of $K[\mathfrak{G}_X]$ is given by convolution:
$$(f_1\cdot f_2)(g)=\sum_{\substack{h\in \mathfrak{G}_X\ \text{s.t.}\ h^{-1}g \\ \text{is well-defined}}} f_1(h)f_2(h^{-1}g).$$
For each $1\leq i\leq d$, consider the characteristic function $1_{x_i}$ of the cylindrical set:
$$ \{u\in X\ |\ u[0]=x_i\}. $$
Then the convolution algebra is generated by these characteristic functions and the shift operator (and its inverse), namely:
$$ K[\mathfrak{G}_X] = K\left< 1_{x_1},\dots,1_{x_d},T^{\pm 1} \right>.$$
Identiying $T$ with the corresponding characteristic function. Notice that $1=\sum_{i=1}^{d} 1_{x_i}$. For more on convolution algebras associated with groupoids arising from subshifts, we refer the reader to \cite{Nekrashevych}, where ring theoretic properties of $K[\mathfrak{G}_X]$ are characterized by means of dynamical properties of $X$. For instance, $K[\mathfrak{G}_X]$ is simple when $X$ is minimal \cite[Theorem~1.2]{Nekrashevych}. The algebra $K[\mathfrak{G}_X]$ is $\mathbb{Z}$-graded by $\deg(1_{x_i})=0,\ \deg(T)=1,\ \deg(T^{-1})=-1$.

Finally, let us view $K[\mathfrak{G}_X]$ as a localization of the monomial algebra $A_X$. Namely, there is an injective (graded) ring homomorphism:
$$ i\colon A_X \hookrightarrow K[\mathfrak{G}_X] $$
given by:
$$ x_i \mapsto 1_{x_i} T.$$
See \cite[Example~4.4.1]{Nekrashevych}, where $A_X$ is denoted $\mathcal{M}_\mathcal{X}$ and $K[\mathfrak{G}_X]$ is denoted $K[\mathfrak{G}_w]$ (where $X$ is the closure of the shift-orbit of $w$).

\subsection{Amenability of convolution algebras}
We are now ready to prove Theorem \ref{convo}.




\begin{proof}[Proof of Theorem \ref{convo}]
By Theorem \ref{mon_amenable}, $A=A_X$ is exhaustively amenable as a right module over itself if and only if for every $D\in \mathbb{N}$ there exists a factor $u$ with a unique length-$D$ right prolongation, say, $ux_{i_1}\cdots x_{i_D}$. Then for each $r<D$ the subspace: $$L=\Span_K\{u,ux_{i_1},\dots,ux_{i_1}\cdots x_{i_r}\}$$ is $(V,1/r)$-invariant for $V=\Span_K\{1,x_1,\dots,x_d\}$, recalling that exhaustive amenability of an algebra as a module over itself can be checked with respect to a fixed generating subspace \cite[Theorem~3.4]{CSSV_Expo}.
We claim that the very same subspace is $(W,2/r)$-invariant for $W=\Span_K\{1,x_1,\dots,x_d,T^{-1}\}$. To this end, it suffice to show that $LT^{-1}\subseteq L+L_0$ where $L_0$ is a $1$-dimensional subspace. Indeed, for each $1\leq s\leq r$, we have: $$(*)\ \ \ ux_{i_1}\cdots x_{i_s}T^{-1}=ux_{i_1}\cdots x_{i_{s-1}}(x_1+\cdots+x_d)T^{-1}=ux_{i_1}\cdots x_{i_{s-1}}\in L$$
(recalling that $T=x_1+\cdots+x_d$; for $s=1$ we formally set $x_{i_{s-1}}$ to be empty). This follows since $ux_{i_1}\cdots x_{i_{s-1}}x_q$ is zero except for $q=i_s$. Hence $LT^{-1}\subseteq L+\Span_K\{uT^{-1}\}$, as required.

Now suppose that $X$ is minimal. Let $M$ be an arbitrary $K[\mathfrak{G}_X]$-module, $V\leq K[\mathfrak{G}_X]$ a finite-dimensional subspace. Elements in $K[\mathfrak{G}_X]$ are spanned by monomials of the form: $$ w = u_0T^{-i_0}u_1 \cdots T^{-i_k}u_{k+1} $$ for some monomials $u_0,\dots,u_{k+1}\in A_X$ and $i_0,\dots,i_k \geq 0$. Let us define the \textit{length} of such a monomial to be $l(w) := |u_0|+\cdots+|u_{k+1}|+i_0+\cdots +i_k$ (where $|u_i|$ is the length with respect to the letters of $X$). Let us define the \textit{degree} of such a monomial to be $\deg(w) := |u_0|+\cdots+|u_{k+1}|-i_0-\cdots -i_k$.

Let $R$ be an upper bound on the lengths of all monomials on which $V$ is supported. Let $\varepsilon > 0$ be given. Let $C=\lceil \frac{2R}{\varepsilon} \rceil$ and $D=C+2R$. By Lemma \ref{combLemma},
there exists a factor $u$ of $X$ with a unique length-$D$ right prolongation, say, $ux_{i_1}\cdots x_{i_D}$. By minimality of $X$, there exists some $N$ such that every length-$N$ monomial in $A$ is divisible by $ux_{i_1}\cdots x_{i_D}$. Given a length-$N$ monomial $v\in A$, let us write $v=v_0ux_{i_1}\cdots x_{i_D}v'$ for some $v_0,v'$.

\textit{Case 1}: For every $\varepsilon > 0$ there exist $\xi\in M$, and a monomial $v\in A$ of length $N$ such that the set:
$$ \{ \xi \cdot v_0ux_{i_1}\cdots x_{i_R},\dots,\xi \cdot v_0ux_{i_1}\cdots x_{i_{R+C}}\} $$ 
is linearly independent (that is, its span -- call it $L'_C$ -- is $(C+1)$-dimensional).

Consider an arbitrary element $f\in V$. Then $f$ is a linear combination of monomials from $K[\mathfrak{G}_X]$ of lengths at most $R$. For each monomial $w \in K[\mathfrak{G}_X]$ of length at most $R$, we have for any $0\leq j \leq C$:
$$ (**) \ \ \ ux_{i_1}\cdots x_{i_{R+j}}w = 0 \ \ \text{or}\ \ ux_{i_1}\cdots x_{i_{R+j}}w = ux_{i_1}\cdots x_{i_{R+j+\deg(w)}} $$ by the same argument as in $(*)$, noticing that $-R\leq \deg(w)\leq R$, so $(**)$ belongs to $L_D:=\Span_K \{\xi \cdot v_0u,\xi\cdot v_0ux_{i_1},\dots,\xi\cdot v_0ux_{i_D}\}$.
It follows that for each $0\leq j\leq C$ we have:
$$ \xi \cdot v_0ux_{i_1}\cdots x_{i_{R+j}}f \in L_D $$
hence $L'_CV\subseteq L_D$. Thus:
\begin{eqnarray*}
\dim_K L'_C V &  \leq & \dim_K L_D\leq D+1 \\ & = & C+2R+1<(1+\varepsilon)(C+1) \\ & = & (1+\varepsilon)\dim_K L'_C
\end{eqnarray*}
so there is a $(V,\varepsilon)$-invariant subspace. Moreover, $\dim_K L'_C\xrightarrow{\varepsilon\rightarrow 0}\infty$, so we get exhaustive amenability of $M$.

\textit{Case 2}: There exists $\varepsilon > 0$ such that for every $\xi \in M$ and for every monomial $v\in A$ of length $N$, the set:
$$ \{ \xi \cdot v_0ux_{i_1}\cdots x_{i_R},\dots,\xi \cdot v_0ux_{i_1}\cdots x_{i_{R+C}}\} $$
is linearly dependent. In particular, either $\xi \cdot v_0ux_{i_1}\cdots x_{i_R} = 0$, so $\xi \cdot v=0$, or for some $1\leq j\leq C$, we can write:
$$ \xi\cdot v_0ux_{i_1}\cdots x_{i_{R+j}} = \xi \cdot \sum_{l=0}^{j-1} \gamma_l v_0ux_{i_1}\cdots x_{i_{R+l}}. $$
So:
$$
\xi\cdot v = \xi\cdot v_0ux_{i_1}\cdots x_{i_D}v' = \xi \cdot \sum_{l=0}^{j-1} \gamma_l v_0ux_{i_1}\cdots x_{i_{R+l}}x_{i_{R+j+1}}\cdots x_{i_D}v'
$$
for some scalars $\gamma_0,\dots, \gamma_{j-1}$.
Hence: $$\xi\cdot v\in \xi\cdot \bigoplus_{i=0}^{N-1} A_i,$$ and it follows that for every $\xi \in M$, we have $\xi \cdot A_N \subseteq \bigoplus_{i=0}^{N-1} A_i$, so $\xi \cdot A$ is a finite-dimensional subspace of $M$. Hence, there exists a non-zero (e.g.~Cayley-Hamilton) polynomial of $T$ which annihilates $\xi \cdot A$.
It follows that for any
$\theta \in \xi\cdot A$ we have:
$$ \theta\cdot T^i=\sum_{p=i+1}^{i'} c_p\ \theta\cdot T^p $$ for some $i,i'$ and scalars $c_p$, so: $$ \theta\cdot T^{-1}= \sum_{p=i+1}^{i'} c_p\theta\cdot T^{p-i-1}\in \xi\cdot A. $$
Hence $\xi\cdot K[\mathfrak{G}_X]=\xi\cdot A$ is a finite-dimensional subspace.

It follows that the $K[\mathfrak{G}_X]$-submodule of $M$ generated by any $0\neq \xi \in M$ is a non-zero, finite-dimensional subspace, which is thus $(V,\varepsilon)$-invariant (for any $\varepsilon>0$).
It follows that $M$ is an exhaustively amenable right $K[\mathfrak{G}_X]$-module.
\end{proof}






Finally, we observe that one cannot hope for the converse implication of the above theorem to hold, namely, $K[\mathfrak{G}_X]$ might be an exhaustively amenable module over itself while $A_X$ is not.
Specifically, consider Example \ref{asymmetry}: it gives a transitive subshift whose monomial algebra $A_X$ is an exhaustively amenable module over itself on one side but not on the other side. However, the convolution algebra $K[\mathfrak{G}_X]$ is isomorphic to its opposite $K[\mathfrak{G}_X]^{\text{op}}$ by the involution $T^*=T^{-1},\ 1_{x_i}^*=1_{x_i}$ (in fact, this is the transpose involution when $K[\mathfrak{G}_X]$ is considered as a ring of row and column finite matrices, as in \cite{Nekrashevych}), so its (exhaustive) amenability as a module over itself is left-right invariant.

\section{Free subalgebras} \label{sec:free}

A group which contains a nonabelian free subgroup is not amenable; Elek \cite{Elek_skew} proved that if a division algebra $D$ contains a non-amenable division subalgebra $E$, then $D$ is non-amenable is well. This cannot be extended to arbitrary algebras: the quotient division algebra of the first Weyl algebra has finite GK-transcendence degree, hence it is amenable, yet contains noncommutative free subalgebras (see \cite{Elek_skew}).
On the other hand, it is well known that there exist non-amenable groups which contain no free subgroups (e.g.~see \cite{Osin_Burnside}).

\subsection{Non-amenable monomial algebras contain free subalgebras} 

We are finally ready to prove Theorem \ref{no free amenable}.


\begin{proof}[Proof of Theorem \ref{no free amenable}]
Let $A$ be a monomial algebra with generators $\Sigma=\{x_1,\dots,x_d\}$ over an infinite field $K$, which is not exhaustively amenable as a right module over itself. If $A$ contained infinitely many non-right prolongable monomials, their span would form an invariant subspace, so $A$ would be exhaustively amenable as a module over itself. If there are only finitely many non-right prolongable monomials, there are only finitely many monomials which are not infinitely right prolongable. Let $N\triangleleft A$ be the span of all of the monomials in $A$ which are not infinitely right prolongable; thus by Lemma \ref{A/N}, $A/N$ is a prolongable monomial algebra which is not exhaustively amenable as a module over itself. Hence we may assume that $A$ itself is right prolongable, since free subalgebras lift from homomorphic images.

Suppose that $A$ is right prolongable. By Theorem \ref{mon_amenable} we deduce that for some $D\in \mathbb{N}$, every monomial $0\neq w\in A$ has at least two distinct length-$D$ right prolongations.

Let $\{\alpha_u\}_{|u|=D}$ be a set of distinct non-zero scalars from $K$. Consider:
$$ x=\sum_{|u|=D} u,\ \ \ y=\sum_{|u|=D} \alpha_u u $$
We claim that $K\left<x,y\right>\subseteq A$ is free.
Assume on the contrary that there is a non-commutative relation between $x,y$. Since $x,y$ are homogeneous of the same degree $D$, we may assume that there is a homogeneous relation, say, 
$$ \sum_{|v|=l} c_v v(X,Y) $$
which vanishes under the substitution $X\mapsto x,\ Y\mapsto y$. Assume that this is a non-trivial relation of minimum possible degree.
This can be re-written as:
$$ \sum_{w\in \{X,Y\}^{l-1}} \lambda_w w(x,y)x = \sum_{w\in \{X,Y\}^{l-1}} \lambda'_w w(x,y)y $$
where here $w$ ranges over length-$(l-1)$ monomials in $X,Y$, substituting $X\mapsto~x,\ Y\mapsto~y$, and at least one of the coefficients is non-zero (otherwise that is not a non-trivial relation). Each $w(x,y)$ can be written as a linear combination of monomials in $x_1,\dots,x_d$:
$$ w(x,y)=\sum_{m\in \Sigma^{D(l-1)}} \beta_{w,m} m. $$
Rewriting $x,y$ by means of $x_1,\dots,x_d$:
\begin{equation}\label{eq:beta} \sum_{\substack{w\in \{X,Y\}^{l-1},\ u\in \Sigma^D \\ m\in \Sigma^{D(l-1)}}} \lambda_w \beta_{w,m} mu = \sum_{\substack{w\in \{X,Y\}^{l-1},\ u\in \Sigma^D \\ m\in \Sigma^{D(l-1)}}} \lambda'_w \beta_{w,m} \alpha_u mu.\end{equation}
Let $\gamma_m = \sum_{w\in \{X,Y\}^{l-1}} \lambda_w \beta_{w,m}$ and $\gamma'_m = \sum_{w\in \{X,Y\}^{l-1}} \lambda'_w \beta_{w,m}$. Thus Equation (\ref{eq:beta}) becomes:
\begin{equation}\label{eq:gamma} 
\sum_{\substack{u\in \Sigma^D \\ m\in \Sigma^{D(l-1)}}} \gamma_m mu = \sum_{\substack{u\in \Sigma^D \\ m\in \Sigma^{D(l-1)}}} \alpha_u  \gamma'_m mu. \end{equation}
We claim that there exists $m\in \Sigma^{D(l-1)}$ with $\gamma_m\neq 0$. Otherwise, $\sum_{w\in \{X,Y\}^{l-1}} \lambda_w w(x,y) = 0$ is a homogeneous relation of smaller degree, hence it must be trivial; but then also $\sum_{w\in \{X,Y\}^{l-1}} \lambda'_w w(x,y)y=0$, and since $y$ is regular ($A$ being a prolongable monomial algebra and all $\alpha_u$ being non-zero) we get $\sum_{w\in \{X,Y\}^{l-1}} \lambda'_w w(x,y)=0$, a homogeneous relation of smaller degree, hence trivial. It follows that $ \sum_{|v|=l} c_v v(X,Y) = 0 $ is the trivial relation, a contradiction.

Pick $m\in \Sigma^{D(l-1)}$ with $\gamma_m\neq 0$.
Now by the non-amenability assumption, together with Theorem \ref{mon_amenable}, we know that $m$ has two distinct length-$D$ right prolongations, say, $u_1,u_2$ such that $mu_1,mu_2$ are non-zero. Since $A$ is a monomial algebra, it follows from Equation (\ref{eq:gamma}) that the coefficients of these monomials on both sides of the equation must coincide:
\begin{eqnarray*}
\gamma_m & = & \alpha_{u_1} \gamma'_m \\
\gamma_m & = & \alpha_{u_2} \gamma'_m.
\end{eqnarray*}
Since $\gamma_m \neq 0$, we obtain that $\alpha_{u_1}=\alpha_{u_2}$, contradicting the way we picked the scalars $\{\alpha_u\}_{|u|=D}$.
Therefore $K\left<x,y\right>\subseteq A$ is a free subalgebra, as claimed.
\end{proof}


\begin{cor}
Let $X$ be a subshift. If the convolution algebra $K[\mathfrak{G}_X]$ over an infinite field is non-(exhaustively amenable) as a module over itself then it contains a noncommutative free subalgebra.
\end{cor}

\begin{proof}
This follows from Theorem \ref{convo} and Theorem \ref{no free amenable}.
\end{proof}

\subsection{A monomial algebra of exponential growth with no free subalgebras}

Let: 
\begin{eqnarray*}
S & = & \bigcup_{k=1}^{\infty} \{10^kn+k-1\ |\ n=1,2,\dots\} \\ & = & \{10,20,\dots, 100, 101, 110,\dots, 1000, 1001,1002, 1010,\dots\}.
\end{eqnarray*}
And construct a monomial algebra as follows. Throughout this section, we let: \begin{equation}
\label{eq:AA}
A=K\left<x_1,x_2,x_3\right>/I,
\end{equation} where $I$ is generated by all monomials which are not factors of any infinite word of the form:
\begin{equation}
\label{eq:AI} x_{\epsilon_1} x_{\epsilon_2} x_{\epsilon_3} \cdots \ \ \ \text{such that}\ \epsilon_i=1\iff i\in S. \end{equation}
Hence $A$ is spanned by the set of monomials which appear as factors of at least one of the aforementioned infinite words.  Throughout this section we let $T$ denote this set of monomials and we let $T(n)$ denote the set of length-$n$ monomials from $T$.

\begin{lem} \label{Growth_lemma}
Let $f\colon \mathbb{N} \rightarrow \mathbb{N}$ be a non-decreasing function.
Let $(a_n)_{n=1}^{\infty}$ be a non-decreasing, unbounded sequence of natural numbers such that $a_{n+1}\leq Ca_n$ for all $n\in \mathbb{N}$ and some $C>1$. If $f(a_n)\geq \alpha^{a_n}$ for some $\alpha>1$ and all $n\in \mathbb{N}$, then $f(n)\geq \beta^n$ for $\beta=\alpha^{1/C}>1$ and all $n\geq a_1$.
\end{lem}
\begin{proof}
Let $f,(a_n)_{n=1}^{\infty},\alpha,C,\beta$ be as in the statement.
Given $n\geq a_1$, take $k$ to be maximum integer such that $a_k\leq n$. Notice that by assumption, $n<a_{k+1}\leq Ca_k$. Hence:
$$ f(n)\geq f(a_k) \geq \alpha^{a_k}\geq \alpha^{n/C}=\beta^n $$
as claimed.
\end{proof}

\begin{prop} \label{Example_10} Let $A$ be as in Equations (\ref{eq:AA}) and (\ref{eq:AI}). Then $A$ is a monomial algebra of exponential growth that contains no noncommutative free subalgebras.
\end{prop}

\begin{proof}
First, to analyze the growth of $A$, we have to show that $|T(n)|\geq \beta^n$ for some $\beta>1$ and $n\gg 1$.

Let $a_n=10^n-1$. 
It is easy to see that $a_{n+1}\leq 11a_n$ for all $n\in \mathbb{N}$.
Consider all length-$a_n$ words $x_{\epsilon_1} x_{\epsilon_2} \cdots x_{\epsilon_{a_n}}$ such that $\epsilon_i=1\iff i\in S$.
The number of indices $1\leq i \leq a_n$ which belong to $S$ is at most:
$$ |S\cap [1,a_n]| \leq a_n \left( \frac{1}{10}+\frac{1}{100}+\cdots \right) = \frac{10^n-1}{9}.$$
Hence, $T(a_n)$ contains at least $2^{a_n-\frac{10^n-1}{9}}=(2^{8/9})^{a_n}$ words (substituting $x_2,x_3$ for each index not participating in $S$). Thus by Lemma \ref{Growth_lemma}, the sequence $|T(n)|$ grows exponentially.

We next claim that the ideal $\left<x_2,x_3\right>\triangleleft A$ is locally nilpotent. Indeed, let $u\in A$ be a non-zero monomial of length $|u| = 2\cdot 10^k$. Then: $$ u=x_{\epsilon_{d+1}}x_{\epsilon_{d+2}}\cdots x_{\epsilon_{d+2\cdot 10^k}} $$
for some $d \geq 0$. But $\{d+1,d+2,\dots, d+2\cdot 10^k\}$ contains a subset of the form: $$ \{n\cdot 10^k,n\cdot 10^k+1,\dots, n\cdot 10^k+(k-1)\}, $$ so $u$ has a factor of the form $x_1^k$. 
Let $B\subseteq \left<x_2,x_3\right>$ be a finitely generated (non-unital) subalgebra of the ideal generated by $x_2,x_3$ in $A$. By enlarging $B$ if necessary, we may assume that it is generated by a finite number of monomials each containing an $x_2$ or an $x_3$. Let $k$ be greater than twice the maximum length of all of these generating monomials; it follows that every monomial from $B$ of length at least $2\cdot 10^k$ vanishes, since otherwise it must contain $x_1^k$ as a factor, which implies that at least one of the monomials generating $B$ is a power of $x_1$, a contradiction. Therefore $B$ is nilpotent.

Now assume to the contrary that $A$ contains a noncommutative free subalgebra on two generators $a,b$. Then $[a,b]=ab-ba$ vanishes modulo $\left<x_2,x_3\right>$, as $A/\left<x_2,x_3\right>\cong K[x_1]$ is commutative, so $[a,b]\in \left<x_2,x_3\right>$ and consequently it is nilpotent, contradicting the assumption that $K\left<a,b\right>$ is free.
\end{proof}

\section*{Acknowledgement}

We thank the anonymous referee for many helpful comments and suggestions and in particular for improvements to Theorem \ref{convo}.

\end{document}